\documentclass[twoside, 11pt]{entics}
\usepackage{enticsmacro}
\usepackage{mathtools}
\usepackage{xcolor}
\usepackage{float}
\usepackage{tikz, tikz-cd, mathtools, amssymb, stmaryrd}
\usepackage[mode=buildmissing]{standalone}

\usepackage{mathpartir}

\def\conf{MFPS 2026}
\volume{NN}

\def\lastname{Lynch and Rischel et. al.}

\begin{document}

\begin{frontmatter}
  \title{Clock systems for stochastic and non-deterministic categorical systems theories}
  \author{Owen Lynch}
  \author{David Jaz Myers}
  \author{Eigil Rischel}
  \author{Sam Staton}
\end{frontmatter}

\maketitle\section{Introduction}\label{rt-001D}\subsection{Overview}\label{rt-0011}
One of the pervasive insights of category theory is that many interesting functors are \emph{representable}. A functor \(F \colon  \mathsf {C} \to  \mathsf {Set}\) is (covariantly) representable if \(F(X) \cong  \mathsf {C}(Y, X)\) for some fixed \(Y \colon  \mathsf {C}\) (we call \(Y\) the representing object for \(F\)). In systems theory, we often study \emph{behaviors} of systems via functors to \(\mathsf {Set}\); the domain of the functor is some category of systems, and we send each system to its set of possible behaviors. Very often, these “behavior functors” end up being representable, and we call the systems which represent these behavior functors \emph{clock systems}.

The story becomes more complicated when we consider \emph{open systems} and the operation of composition for open systems. Now instead of a category of systems, we have a category of systems indexed over a double category of interfaces, where the two directions in the double category represent \emph{composition patterns} and \emph{interface morphisms} respectively \cite{myers-2023-categorical}. Now clock systems are desirable because behavior functors are much more complex (being a functor of indexed categories), and defining a behavior functor as the representable functor for a fixed system greatly simplifies the task of checking all of the various conditions that must be satisfied for such a functor.

It was previously not known whether systems theories for stochastic and nondeterministic systems supported clock systems. In this paper we demonstrate clock systems which replicate classical notions of behavior for stochastic systems, and analogues to these clock systems suggest new notions of behavior for nondeterministic systems.
\subsection{Moore machines for a monad, morphisms of Moore machines, and clocks}\label{rt-001A}
In this section, we review the main concepts of the paper in a simplified setting.

We consider Moore machines whose branching behaviour is governed by a monad, which is a special case of categorical systems theory. We explain how a “clock object” can be regarded as representing the trajectories of a system, and we phrase our main results in this way.

Recall that a Moore machine comprises a set of states \(X\), an input alphabet \(A^-\), an output alphabet \(A^+\), and output and transition functions:
\begin{equation} f:X\to  A^+, \quad  f^\sharp :X\times  A^-\to  X\text . \end{equation}
An old and central idea of coalgebraic systems (e.g. \cite{rutten-2000-universal}) is to generalize this by considering a category \(\mathcal {S}\) with finite products and
a strong monad \(M\) on it. Then an \(M\)-Moore machine \((X,A^+,A^-,f,f^\sharp)\) is given by three objects of \(\mathcal {S}\): states \(X\), input alphabet \(A^-\), and output alphabet \(A^+\), and transition morphisms in \(\mathcal {S}\) too:
\begin{equation} f:X\to  A^+,\quad  f^\sharp :X\times  A^-\to  M(X)\text . \end{equation}
Four crucial examples illustrate how this subsumes various traditional notions of machine:
\begin{itemize}\item{}An ordinary Moore machine arises when \(\mathcal {S}=\mathsf {Set}\) and \(M\) is the identity monad;
\item{}When \(\mathcal {S}=\mathsf {Set}\) and \(P\) is the powerset monad, we have a non-deterministic finite automaton together with outputs at each state, as might be called a “non-deterministic Moore machine”;
\item{}When \(\mathcal {S}=\mathsf {Set}\) and \(D\) is the finite probability distributions monad, we have a probabilistic finite automaton together with outputs at each state, as might be called a “stochastic Moore machine” (see e.g. \cite{sokolova-2011-probabilistic}).
\item{}When \(\mathcal {S}=\mathsf {Meas}\), the category of measurable spaces, and \(G\) is the Giry monad for probability, we have a probabilistic automaton together with outputs at each state. Now where the states, inputs, and outputs may be uncountable spaces such as \(\mathbb {R}\) and we have a full range of continuous probability distributions, including Gaussian distributions. Structures such as these appear in Markov processes and Markov decision processes.\end{itemize}
A morphism of \(M\)-Moore machines
\begin{equation}(x,a,a^\flat ): (X_1,A_1,f_1)\to  (X_2,A_2,f_2)\end{equation}
is given by functions \(x : X_1 \to  X_2\), \(a:A^+_1\to  A^+_2\), \(a^\flat :A^+_1\times  A^-_1\to  A^-_2\), such that the following diagrams commute.
\begin{center}
\begin {tikzcd}
 {X_1} & {A^+_1} \\
 {X_2} & {A^+_2}
 \arrow ["{f_1}", from=1-1, to=1-2]
 \arrow ["{f_2}"', from=2-1, to=2-2]
 \arrow ["{x}"', from=1-1, to=2-1]
 \arrow ["{a}", from=1-2, to=2-2]
\end {tikzcd}
\quad 
\begin {tikzcd}
 {X_1\times  A^-_1} & {M(X_1)} \\
 {X_1 \times  A^+_1 \times  A^-_1} \\
 {X_2\times  A^-_2} & {M(X_2)}
 \arrow ["{f_1^\sharp }", from=1-1, to=1-2]
 \arrow ["{M(x)}", from=1-2, to=3-2]
 \arrow ["{(\mathrm {id},f_1)\times  A^-_1}"', from=1-1, to=2-1]
 \arrow ["{x\times  a^\flat }"', from=2-1, to=3-1]
 \arrow ["{f_2^\sharp }"', from=3-1, to=3-2]
\end {tikzcd}
\end{center}
We thus form categories \(\mathsf {Sys}(M)\) of Moore machines for each strong monad \(M\). 

The literature already suggests notions of clock and trajectory for different kinds of \(M\)-Moore machines. In an ordinary Moore machine, a trajectory is a sequence of inputs, states, and outputs, that arises from stepping through the machine, with a clock counting the steps. In a Markov process, a trajectory is a sequence of random inputs, states and outputs; note that these are random variables in this case, and are indexed by a filtration, which plays the role of a clock. The different notions of clock and trajectory are ad hoc, and vary pragmatically for different \(M\)-Moore machines. Nonetheless, we show that for three of the above cases they are representable, in the following sense. There is a “clock system” \(M\)-Moore machine \(C \in  \mathsf {Sys}(M)\) and the representable hom-functor
\begin{equation}\mathsf {Sys}(M)(C,-):\mathsf {Sys}(M)\to  \mathsf {Set}\end{equation}
associates each machine with its trajectories.
\begin{itemize}\item{}In the deterministic case, the universal trajectory \(C\) is a clock machine, that updates a counter. This is already well-known in double-categorical systems theory.
\item{}In the measurable stochastic case, the universal trajectory \(C\) is a machine that produces a growing stream of random seeds. This is the first contribution of this article.
\item{}In the non-deterministic case, the universal trajectory \(C\) is a machine produces a growing stream of nondeterministic seeds. We also consider a generalization to decorated paths through a graph. This is the second contribution of this article.\end{itemize}
In the rest of the paper, we will elaborate on this simplified account and make it work within a double categorical picture, which accounts for the compositionality of Moore machines and consequently, the compositionality of their behaviors (\cite{myers-2021-double}, \cite{myers-2023-categorical}, \cite{myers-2025-doctrines}). In \cite{boccali-2023-bicategories} they also consider compositionality of Moore machines; the difference between this approach and the double categorical approach is that we consider morphisms between systems which may change the interface. Such morphisms are crucial for clock systems to exist.
\subsection{Organization of the rest of the paper}\label{rt-001C}
This paper is organized as follows. First in \cref{rt-0008} we give a general overview of the type of systems theory we are using (systems theories constructed from \emph{tangencies}, or in other words systems theories of generalized Moore machines). Then in \cref{rt-0009} we review the theory of representable behaviors for discrete-time deterministic Moore machines. This serves \cref{rt-000A}, which is where \cref{rt-000Q} is found characterizing behaviors of discrete-time stochastic Moore machines. In \cref{rt-000V}, we give the analogous construction for discrete-time non-deterministic Moore machines, which turns out to be quite simple and admits an interesting generalization to non-linear time. Finally, we give some corollaries of representable behavior in \cref{rt-0015}, and conclude with some proposed future extensions in \cref{rt-000U}.
\section{Systems theories via tangencies}\label{rt-0008}
In this section, we review the construction of systems theories via tangencies. Most of this can be found in \cite{myers-2023-categorical}, but we adopt the new term “tangency” from \cite{myers-2025-doctrines}.

We are mainly interested in the systems theory presented in \cref{rt-000H}, however to understand behaviors in this systems theory, it is useful to have the contrast class of \cref{rt-000G}; thus we briefly review the general construction so that the analogy between the two systems theories can be made clearer.

\begin{construction}[{Lenses and charts from an indexed category}]\label{rt-000D}
This construction is a recap of section 4 of \cite{myers-2021-double}.

Let \(\mathcal {S}\) be a category and \(\mathcal {B} \colon  \mathcal {S}^\mathrm {op} \to  \mathsf {Cat}\) an indexed category over \(\mathcal {S}\). The attitude we will take towards \(\mathcal {S}\) and \(\mathcal {B}\) is that \(\mathcal {S}\) is a category of spaces, and \(\mathcal {B}(X)\) is a category of bundles over \(X \colon  \mathcal {S}\). We will use the notation \(\mathcal {B}(f) = f^\ast  \colon  \mathcal {B}(B) \to  \mathcal {B}(A)\), in line with the view of \(f^\ast \) as being some kind of “pullback operation”.

We may then construct a double category \(\iint  \mathcal {B}\) in the following way.
\begin{itemize}\item{}The objects of \(\iint  \mathcal {B}\) are pairs \((A^+ \colon  \mathcal {S}, A^- \colon  \mathcal {B}(A))\). We call these \textbf{arenas}, and in systems theory, these play the role of \emph{interfaces}.
\item{}The tight arrows of \(\iint  \mathcal {B}\) from \(A_1\) to \(A_2\) are pairs \((a \colon  A_1^+ \to  A_2^+, a^\flat  \colon  A_1^- \to  a^\ast (A_2^-))\). We call these \textbf{charts}, and in systems theory these play the role of \emph{interface morphisms}.
\item{}The loose arrows of \(\iint  \mathcal {B}\) from \(A\) to \(B\) are pairs \((f \colon  A^+ \to  B^+, f^\sharp  \colon  f^\ast (B^-) \to  A^-)\). We call these \textbf{lenses}, and in systems theory these play the role of \emph{composition patterns}.
\item{}A 2-cell exists inside a square
  \begin{center}
\begin {tikzcd}
	{A_1} & {B_1} \\
	{A_2} & {B_2}
	\arrow ["{f_1}", "\shortmid "{marking}, from=1-1, to=1-2]
	\arrow ["a"', from=1-1, to=2-1]
	\arrow ["b", from=1-2, to=2-2]
	\arrow ["{f_2}"', "\shortmid "{marking}, from=2-1, to=2-2]
\end {tikzcd}
\end{center}

precisely when the following two diagrams commute

  \begin{center}
\begin {tikzcd}
	{A_1^+} & {B_1^+} & {A_1^-} && {f_1^\ast (B_1^-)} \\
	{A_2^+} & {B_2^+} & {a^\ast (A_2^-)} & {a^\ast (f_2^\ast (B_2^-))} & {f_1^\ast (b^\ast (B_2^-))}
	\arrow ["{f_1}", from=1-1, to=1-2]
	\arrow ["a"', from=1-1, to=2-1]
	\arrow ["b", from=1-2, to=2-2]
	\arrow ["{a^\flat }"', from=1-3, to=2-3]
	\arrow ["{f_1^\sharp }"', from=1-5, to=1-3]
	\arrow ["{f_1^\ast (b^\flat )}", from=1-5, to=2-5]
	\arrow ["{f_2}"', from=2-1, to=2-2]
	\arrow ["{a^\ast (f_2^\sharp )}", from=2-4, to=2-3]
	\arrow [equals, from=2-4, to=2-5]
\end {tikzcd}
\end{center}

We call these \textbf{lens-chart squares}, and in the systems theory these play the role of \emph{morphisms of composition patterns}.
\end{itemize}\end{construction}

\begin{example}[{Bundles of arrows}]\label{rt-000F}
For a category \(\mathcal {S}\) with finite limits, we may define \(\mathcal {B} \colon  \mathcal {S}^\mathrm {op} \to  \mathsf {Cat}\) by \(A \mapsto  \mathcal {S} /A\), with action on morphisms defined by pullback.

We can also consider sending \(A\) to some subcategory of \(\mathcal {S} /A\), such as the subcategory \(\mathsf {Ctx}_A\) of product projections \(A \times  F \to  A\); note that this works even if \(\mathcal {S}\) doesn't have all finite limits but does have products. We can identify an object \(A \times  B \to  A\) in \(\mathsf {Ctx}_A\) with its fiber \(B\), but importantly, a morphism \(B \to  B'\) in \(\mathsf {Ctx}_A\) corresponds to a morphism \(A \times  B \to  B'\) in \(\mathcal {S}\), not just a morphism \(B \to  B'\).
\end{example}

\begin{definition}[{Tangency}]\label{rt-000C}
Suppose that \((\mathcal {S}, \mathcal {B} \colon  \mathcal {S}^\mathrm {op} \to  \mathsf {Cat})\) is an indexed category. Then a \textbf{section} of \(\mathcal {B}\) is a section of the fibration \(\int  \mathcal {B} \to  \mathcal {S}\). Concretely, for each \(X \colon  \mathcal {S}\), \(T(X) \colon  \mathcal {B}(X)\), and for \(f \colon  X \to  Y\), \(T(f) \colon  T(X) \to  f^\ast (T(Y))\). A \textbf{tangency} is an indexed category \((\mathcal {S}, \mathcal {B})\) equipped with a section \(T\).
\end{definition}

\begin{construction}[{The systems theory associated with a tangency}]\label{rt-000E}
Given a tangency \((\mathcal {S},\mathcal {B},T)\), there is an associated \emph{systems theory}. Technically speaking, there is a bifunctor from the bicategory of tangencies to the bicategory of systems theories, but rather than fully unpacking that statement, we will give a conceptual overview.

A systems theory consists of:
\begin{itemize}\item{}interfaces
	\item{}interface morphisms
	\item{}composition patterns
	\item{}composition pattern morphisms
	\item{}systems
	\item{}system morphisms\end{itemize}
which all together assemble into a mathematical structure which determines how they interact. For instance, each system has an interface. One can use composition patterns to compose systems. Morphisms of systems restrict to morphisms of their interfaces. The interested reader can refer to \cite{libkind-2025-towards}. Without further ado:
\begin{itemize}\item{}An \textbf{interface} is an arena for \((\mathcal {S},\mathcal {B})\).
\item{}An \textbf{interface morphism} is a chart for \((\mathcal {S},\mathcal {B})\).
\item{}A \textbf{composition pattern} is a lens for \((\mathcal {S}, \mathcal {B})\).
\item{}A \textbf{composition pattern morphism} from \(f_1 \colon  A_1 \mathrel {\mkern 3mu\vcenter {\hbox {$\shortmid $}}\mkern -10mu{\to }} B_1\) to \(f_2 \colon  A_2 \mathrel {\mkern 3mu\vcenter {\hbox {$\shortmid $}}\mkern -10mu{\to }} B_2\) is a pair of charts \(a \colon  A_1 \to  A_2\), \(b \colon  B_1 \to  B_2\) such that \(f_1,f_2,a,b\) form a lens-chart square for \((\mathcal {S}, \mathcal {B})\).
\item{}A \textbf{system} on an interface \(A\) consists of an object \(X \colon  \mathcal {S}\) along with a lens \(TX \to  A\).
\item{}A \textbf{system morphism} between a system \(f_1 \colon  TX_1 \to  A_1\) and a system \(f_2 \colon  TX_2 \to  A_2\) consists of a morphism \(x \colon  X_1 \to  X_2\) in \(\mathcal {S}\), an interface morphism \(a \colon  A_1 \to  A_2\), and lens-chart square

  \begin{center}
\begin {tikzcd}
	{TX_1} & {A_1} \\
	{TX_2} & {A_2}
	\arrow ["{f_1}", "\shortmid "{marking}, from=1-1, to=1-2]
	\arrow ["Tx"', from=1-1, to=2-1]
	\arrow ["a", from=1-2, to=2-2]
	\arrow ["{f_2}"', "\shortmid "{marking}, from=2-1, to=2-2]
\end {tikzcd}
\end{center}\end{itemize}
For now, we will primarily focus on the category of systems and system morphisms. To be more precise, let \(\mathsf {Sys}(\mathcal {S}, \mathcal {B}, T)\) be the category where the objects are pairs of an interface \(A\) and a system \(TX \to  A\), and a morphism is as defined above.
\end{construction}

\begin{example}[{The tangency for discrete-time deterministic Moore machines}]\label{rt-000G}
Let \(\mathcal {S} = \mathsf {Set}\), and let \(\mathcal {B}(A) = \mathsf {Ctx}_A\) (as defined in \cref{rt-000F}), with tangent bundle given by \(TX = (X,X)\).

Then a system for \((\mathcal {S}, \mathcal {B}, T)\) on an arena \(A = (A^+, A^-)\) consists of a set \(X\), a function \(f \colon  X \to  A^+\) (which we think of as the “output function” of the system), and a function \(f^\sharp  \colon  X \times  A^- \to  X\) (which we think of as the “update function” of the system).

This recovers the classical notion of Moore machine with input alphabet \(A^-\) and output alphabet \(A^+\).
\end{example}

\begin{example}[{The tangency for discrete-time stochastic Moore machines}]\label{rt-000H}
We will now consider a variant of \cref{rt-000G} in which the update can be stochastic. To do this, let \(\mathcal {S} = \mathsf {Meas}\), and define \(\mathcal {B}(A) = \mathsf {Ctx}_A\). However, this time let the tangent bundle functor be defined by \(X \mapsto  (X,\Delta  X)\), where \(\Delta \) is the Giry monad on the category of measurable spaces. A system for this tangency on an arena \(A\) consists of a measurable space \(X\) along with a measurable map \(X \to  A^+\) and a stochastic kernel \(X \times  A^- \to  X\).

This is a slight variation on what is found in section 2.2 of \cite{myers-2023-categorical}; the difference is here we only allow the systems to be stochastic, not the lenses nor the charts. This is necessary for \cref{rt-000Q} to work out correctly.
\end{example}

\begin{example}[{The tangency for discrete-time nondeterministic Moore machines}]\label{rt-000W}
The tangency for discrete-time nondeterministic Moore machines is almost the same as \cref{rt-000H}, except instead we take \(\mathcal {S} = \mathsf {Set}\) and we replace the Giry monad with the powerset monad \(P\). That is, \(\mathcal {B}(A) = \mathsf {Ctx}_A\), and the tangent bundle \(T\) is defined by \(TX = (X, PX)\).
\end{example}
\section{Representable behaviors for discrete-time deterministic Moore machines}\label{rt-0009}
In this section, we review the classical definition of “behavior” for a discrete-time deterministic Moore machine, and show how this classical definition may be recovered as a representable functor.

\begin{definition}[{Behaviors of discrete-time deterministic Moore machines}]\label{rt-000B}
Given a system \(f \colon  TX \mathrel {\mkern 3mu\vcenter {\hbox {$\shortmid $}}\mkern -10mu{\to }} A\) from the tangency of \cref{rt-000G}, a \textbf{behavior} consists of sequences \(x_1,x_2,\ldots  \in  X\), \(a_1^+,a_2^+,\ldots  \in  A^+\) and \(a_1^-,a_2^-,\ldots  \in  A^-\) such that:
\begin{equation} f^+(x_i) = a_i^+ \end{equation}\begin{equation} f^-(x_i, a_i^-) = x_{i+1} \end{equation}
We define \(B(f)\) to be the set of behaviors for the system \(f\). We now claim that \(B(f)\) is a functor from \(\mathsf {Sys}(\mathsf {Set}, \mathsf {Ctx}, T)\) to \(\mathsf {Set}\), where \((\mathsf {Set}, \mathsf {Ctx}, T)\) is the tangency of \cref{rt-000G}. To see this, suppose that
\begin{center}
\begin {tikzcd}
	{TX_1} & {A_1} \\
	{TX_2} & {A_2}
	\arrow ["{f_1}", "\shortmid "{marking}, from=1-1, to=1-2]
	\arrow ["Tx"', from=1-1, to=2-1]
	\arrow ["a", from=1-2, to=2-2]
	\arrow ["{f_2}"', "\shortmid "{marking}, from=2-1, to=2-2]
\end {tikzcd}
\end{center}
is a morphism of systems and that \(x_{1,i}, a_{1,i}^+, a_{1,i}-\) is a behavior of \(f_1\). Then we may define new sequences \(x_{2,i} = x(x_{1,i}), a_{2,i}^+ = a^+(a_{1,i}^+), a_{2,i}^- = a^-(a_{1,i}^+, a_{1,i}^-)\), and we claim that these new sequences form a behavior of \(f_2\).

We must now show that \(a_{2,i}^+ = f_2(x_{2,i})\) and that \(x_{2,i+1} = f_2^\sharp (x_{2,i}, a_{2,i}^-)\). These can be proven by unwrapping the definitions and then applying the equalities in the definition of system morphism. For the first equality,
\begin{equation}a_{2,i}^+ = a^+(a_{1,i}^+) = a^+(f_1(x_{1,i})) = f_2(x(x_{1,i})) = f_2(x_{2,i})\end{equation}
For the second equality,
\begin{equation}x_{2,i+1} = x(x_{1,i+1}) = x(f_1^\sharp (x_{1,i}, a_{1,i}^-)) = f_2^\sharp (x(x_{1,i}), a^-(a_{1,i}^+, a_{1,i}^-)) = f_2^\sharp (x_{2,i}, a_{2,i}^-)\end{equation}
We are done.
\end{definition}

Showing that \(B\) is representable begins by defining a clock system.

\begin{definition}[{The clock for discrete-time deterministic Moore machines}]\label{rt-000I}
The clock \(c\) is a system on the interface \((\mathbb {N}, 1)\). Its state space is \(\mathbb {N}\), and the lens \(c \colon  T\mathbb {N} \mathrel {\mkern 3mu\vcenter {\hbox {$\shortmid $}}\mkern -10mu{\to }} (\mathbb {N}, 1)\) is defined as follows.
\begin{equation}c(n) = n\end{equation}\begin{equation}c^\sharp (n,\ast ) = n+1\end{equation}
This models an abstract clock which simply “counts up”.
\end{definition}

\begin{proposition}[{Behaviors as maps out of the clock}]\label{rt-000J}
The behavior functor \(B \colon  \mathsf {Sys} \to  \mathsf {Set}\) as constructed in \cref{rt-000B} is represented by maps out of the clock system in \cref{rt-000I}.
\end{proposition}
\begin{proof}
A system morphism from \(c\) to \(f \colon  TX \to  A\) consists of a function \(x \colon  \mathbb {N} \to  X\) and a chart \(a \colon  (\mathbb {N}, 1) \to  A\), along with two commutative squares. We will see that these commutative squares precisely correspond to \(x_i = x(i)\), \(a_i^+ = a^+(i)\), \(a_i^- = a^-(i, \ast )\) being a behavior of the system.

The first commutative square is:
\begin{center}
\begin {tikzcd}
	\mathbb {N} & \mathbb {N} \\
	X & {A^+}
	\arrow [equals, from=1-1, to=1-2]
	\arrow ["x"', from=1-1, to=2-1]
	\arrow ["{a^+}", from=1-2, to=2-2]
	\arrow ["{f^+}"', from=2-1, to=2-2]
\end {tikzcd}
\end{center}
This commutative square corresponds to the condition that \(a_i^+ = f^+(x_i)\).

The second commutative square is:
\begin{center}
\begin {tikzcd}
	{\mathbb {N} \times  1} & \mathbb {N} \\
	{X \times  A^-} & X
	\arrow ["{+1}", from=1-1, to=1-2]
	\arrow ["{x \times  a^-}"', from=1-1, to=2-1]
	\arrow ["x", from=1-2, to=2-2]
	\arrow ["{f^-}"', from=2-1, to=2-2]
\end {tikzcd}
\end{center}
This commutative square corresponds to the condition that \(x_{i+1} = f^-(x_i, a_i^-)\). We are done.
\end{proof}
\section{Representable behaviors for discrete-time stochastic Moore machines}\label{rt-000A}
We can try to define a clock system in the systems theory for the tangency of \cref{rt-000H} in the same way as \cref{rt-0009}, using \(\mathbb {N}\) as the state space. However, the morphisms out of such a clock are much less interesting; they do not represent stochastic behaviors in any reasonable sense.

In this section, we remedy this by first reviewing a good classical notion of behavior for stochastic Moore machines, and then showing that this notion of behavior is captured by a representable functor.

However before we begin, we must address a conflict in notation between probability theory and category theory. In probability theory, it is conventional to denote random variables by uppercase letters \(X, Y, \ldots \). However, in category theory, it is much more natural to denote a random variable as a lowercase letters, as a random variable is a morphism \(x \colon  \Omega  \to  \mathbb {R}\). Moreover, we typically use uppercase letters to denote the objects of a category.

We resolve this conflict in favor of category theory, and use lowercase letters for random variables while discussing probability theory. We hope that the reader will graciously accommodate this eccentricity in service of a greater consistency.

In order to review the classical notion of behavior, we must recall some basic definitions from measure theory; the following four definitions form the basis for saying what it means for a stochastic process to be a Markov process for a particular kernel \(X \to  \Delta (X)\).

\begin{definition}[{Filtration}]\label{rt-000L}
Given a measurable space \((\Omega , \mathcal {F}_\Omega )\), a \textbf{filtration} \(\mathbb {F}\) on \(\Omega \) is an ascending sequence \(\mathcal {F}_1 \subset  \mathcal {F}_2 \subset  \cdots  \subset  \mathcal {F}_\Omega \) of sub-\(\sigma \)-algebras of \(\mathcal {F}_\Omega \).
\end{definition}

\begin{definition}[{Process adapted to a filtration}]\label{rt-000M}
Given a measurable space \((\Omega , \mathcal {F}_\Omega )\) and a filtration \(\mathbb {F} = \mathcal {F}_1 \subset  \mathcal {F}_2 \subset  \cdots \) on it, a sequence of \(X\)-valued variables \(x_1, x_2, \ldots  \colon  \Omega  \to  X\) is \textbf{adapted} to \(\mathbb {F}\) if \(x_i\) is measurable with respect to \(\mathcal {F}_i\) for every \(i\).
\end{definition}

\begin{definition}[{Conditional expectation}]\label{rt-000K}
Suppose that \((\Omega , \mathcal {F}_\Omega , P)\) is a probability space. Let \(\mathcal {G}\) be a sub-\(\sigma \)-algebra of \(\mathcal {F}_\Omega \), and finally suppose that \(x \colon  \Omega  \to  \mathbb {R}\) is measurable with respect to \(\mathcal {F}_\Omega \) (but not necessarily \(\mathcal {G}\)).

Then a conditional expectation \(\mathbb {E}_P(x|\mathcal {G})\) is defined to be any \(\mathcal {G}\)-measurable random variable such that for all \(U \in  \mathcal {G}\),
\begin{equation}\int _U \mathbb {E}(x|\mathcal {G}) \,\mathrm {d}\mu  = \int _U x \,\mathrm {d}\mu \end{equation}
One can prove that conditional expectations always exist and are \(P\)-almost unique, and that is what justifies the phrase “the” conditional expectation.
\end{definition}

\begin{definition}[{Conditional probability}]\label{rt-000O}
Suppose that \((\Omega , \mathcal {F}_\Omega , P)\) is a probability space. Let \(\mathcal {G}\) be a sub-\(\sigma \)-algebra of \(\mathcal {F}_\Omega \), and let \(x \colon  \Omega  \to  X\) be a \(X\)-valued random variable on \(\Omega \), where \((X, \mathcal {F}_X)\) is a measurable space.

Then given any \(V \in  \mathcal {F}_X\), let \(\chi _{V} \circ  x \colon  \Omega  \to  \mathbb {R}\) be the composite \(\Omega  \xrightarrow {x} X \xrightarrow {\chi _{V}} \mathbb {R}\) where \(\chi _{V}\) is the indicator function for \(V\); this always produces a \(\mathcal {G}\)-measurable random variable \(\mathbb {E}_P(\chi _{V} \circ  x|\mathcal {G}) \colon  \Omega  \to  \mathbb {R}\). These assemble into a \(\mathcal {G}\)-measurable function \(P(x|\mathcal {G}) \colon  \Omega  \to  \Delta (X)\), defined up to \(P\)-almost sure equality by
\begin{equation}P(x|\mathcal {G})(\omega  \in  \Omega )(V \in  \mathcal {F}_X) = \mathbb {E}_P(\chi _{V} \circ  x|\mathcal {G})(\omega ).\end{equation}
We call this the conditional probability for \(x\) given \(\mathcal {G}\).
\end{definition}

As promised, this allows us to talk about Markov processes with a law.

\begin{definition}[{Markov process with a law}]\label{rt-001G}

Given a stochastic kernel \(s \colon  X \to  \Delta (X)\), an \(\mathbb {F}\)-adapted process \(x_1, x_2,\ldots  \colon  \Omega  \to  X\) on the probability space \((\Omega , \mathcal {F}_\Omega , P)\) is said to be Markov with law \(s\) if for all \(i \colon  \mathbb {N}\),
\begin{equation}P(x_{i+1}|\mathcal {F}_i)(\omega ) = s(x_i(\omega ))\quad \text {$P$-almost everywhere}.\end{equation}
In other words, the probability distribution over \(x_{i+1}\) conditioned on “the information we have up to time \(i\)” is given by \(s(x_{i})\).
\end{definition}

The extension to stochastic Moore machines (which might be thought of as “open” Markov processes) is straightforward.

\begin{definition}[{Behavior of a discrete-time stochastic Moore machine}]\label{rt-000N}
Let \((\Omega , \mathcal {F}_\Omega , P)\) be a probability space, and let \(\mathbb {F} = \mathcal {F}_1 \subset  \mathcal {F}_2 \subset  \cdots \) be a filtration on it.

Then let \(f \colon  TX \to  A\) be a discrete-time stochastic Moore machine (that is, a system for the tangency of \cref{rt-000H}). A behavior for \(f\) with respect to \((\Omega , \mathcal {F}_\Omega , P, \mathbb {F})\) consists of \(\mathbb {F}\)-adapted processes \(x_1, x_2, \ldots  \colon  \Omega  \to  X\), \(a_1^+, a_2^+, \ldots  \colon  \Omega  \to  A^+\), and \(a_1^-, a_2^-, \ldots  \colon  \Omega  \to  A^-\), such that:

\begin{equation}a_i^+(\omega ) = f(x_i(\omega ))\end{equation}\begin{equation}P(x_{i+1}|\mathcal {F}_i)(\omega ) = f^\sharp (x_i(\omega ), a^-_i(\omega ))(x_{i+1})\end{equation}
Just as in \cref{rt-000B}, behaviors for a fixed \((\Omega , \mathcal {F}_\Omega , P, \mathbb {F})\) form a functor \(B_{(\Omega , F_\Omega , P, \mathbb {F})}\) from \(\mathsf {Sys}\) to \(\mathsf {Set}\), and the proof is essentially the same.
\end{definition}

Note that a behavior could only be captured by a system morphism out of the deterministic clock of \cref{rt-000I} if \(f^\sharp (x_i, a_i^-)\) were a delta distribution for all \(i\), which is not very interesting.

However, we can choose a different clock system and recover representability of behaviors.

\begin{definition}[{Stochastic clock system}]\label{rt-000P}
Fix a probability space \((\Omega , \mathcal {F}_\Omega , P)\) and a filtration \(\mathbb {F}\) on it. Let
\begin{equation}\tilde {\Omega } = \sum _{i \colon  \mathbb {N}} (\Omega , \mathcal {F}_i)\end{equation}
Then form a system on the interface \((\tilde {\Omega }, 1)\) with state space \(\tilde {\Omega }\) and lens \(c \colon  T\tilde {\Omega } \to  (\tilde {\Omega }, 1)\) defined by
\begin{equation}c((i, \omega )) = (i, \omega )\end{equation}\begin{equation}c^\sharp ((i, \omega ), \ast ) = (i + 1, P(1_\Omega  | \mathcal {F}_i)(\omega )) \end{equation}
where \((i + 1, P(1_\Omega  | \mathcal {F}_i)(\omega ))\) refers to the distribution on \(\tilde {\Omega }\) given by pushing forward \(P(1_\Omega  | \mathcal {F}_i)(\omega )\) along the map \(\Omega  \to  \tilde {\Omega }\), \(\omega  \mapsto  (i+1, \omega )\).

The intuition for this is that \(c^\sharp \) resamples the current \(\omega \), only “remembering” the information contained in \(\mathcal {F}_i\), while deterministically incrementing the time.

As conditional probability is only defined up to almost-sure equality, technically speaking there are many clock systems for a given \((\Omega , \mathcal {F}_\Omega , P)\) corresponding to different choices of \(P(1_\Omega  | \mathcal {F}_i) \colon  \Omega  \to  \Delta  \Omega \). We will continue to say “the clock system” for \((\Omega , \mathcal {F}_\Omega , P)\), and by this we mean some fixed choice of \(P(1_\Omega |\mathcal {F}_i)\) for each \(i\).
\end{definition}

We now build up some theory for characterizing system morphisms out of the clock system.

\begin{remark}[{Adapted processes are maps out of \(\tilde {\Omega }\)}]\label{rt-000S}
Given a measurable space \((\Omega , \mathcal {F}_\Omega )\) and filtration \(\mathbb {F}\), the set of measurable maps from \(\tilde {\Omega }\) to \(X\) is bijective with the set \(X\)-valued processes adapted to \(\mathbb {F}\), by the universal property of the coproduct and the definition of adaptation.
\end{remark}

\begin{lemma}[{Pushforward of conditional probability}]\label{ocl-00CW}
Suppose that \((\Omega , \mathcal {F}_\Omega , P)\) is a probability space and \(\mathcal {F}' \subset  \mathcal {F}_\Omega \) is a sub-\(\sigma \)-algebra of \(\mathcal {F}_\Omega \). Let \(x \colon  \Omega  \to  X\) and \(f \colon  X \to  Y\) be measurable functions. Then
\begin{equation}(\Delta  f)P(x | \mathcal {F}') = P(f \circ  x | \mathcal {F}')\end{equation}
where \(\Delta  f \colon  \Delta  X \to  \Delta  Y\) is the pushforward of probability distributions associated with \(f\).
\end{lemma}

\begin{theorem}[{System morphisms out of the stochastic clock are behaviors}]\label{rt-000Q}
Fix a probability space \((\Omega , \mathcal {F}_\Omega , P)\) and a system \(f \colon  TX \to  A\). Then system morphisms into \(f\) out of the clock \(c \colon  T\tilde {
\Omega } \mathrel {\mkern 3mu\vcenter {\hbox {$\shortmid $}}\mkern -10mu{\to }} (\tilde {\Omega }, 1)\) as defined in \cref{rt-000P} are in canonical bijection with behaviors of \(f\) as defined in \cref{rt-000N}.
\end{theorem}
\begin{proof}
A system morphism from the clock \(c\) into \(f\) consists of a pair \(x \colon  \tilde {\Omega } \to  X, a \colon  (\tilde {\Omega }, 1) \to  A\). By \cref{rt-000S}, this is in bijection with a choice of processes \(x_i \colon  \Omega  \to  X\), \(a^+_i \colon  \Omega  \to  A^+\), \(a^-_i \colon  \Omega  \to  A^-\) adapted to \(\mathbb {F}\).

The first criterion for such processes to form a behavior of \(F\), that \(a^+_i(\omega ) = f(x_i(\omega ))\) is trivially derived from the commutativity of
\begin{center}
\begin {tikzcd}
	{\tilde {\Omega }} & {\tilde {\Omega }} \\
	X & {A^+}
	\arrow [equals, from=1-1, to=1-2]
	\arrow ["x"', from=1-1, to=2-1]
	\arrow ["a", from=1-2, to=2-2]
	\arrow ["f"', from=2-1, to=2-2]
\end {tikzcd}
\end{center}
The second criterion is trickier, as it involves conditional probability. We claim that the commutativity of
\begin{center}
\begin {tikzcd}
	{\tilde {\Omega } \times  1} & {\Delta  \tilde {\Omega }} \\
	{X\times  A^-} & \Delta  X
	\arrow ["{c^\sharp }", from=1-1, to=1-2]
	\arrow ["{x \times  a^\flat }"', from=1-1, to=2-1]
	\arrow ["\Delta  x", from=1-2, to=2-2]
	\arrow ["{f^\sharp }"', from=2-1, to=2-2]
\end {tikzcd}
\end{center}
is equivalent to satisfying the equation
\begin{equation}P(x_{i+1}|\mathcal {F}_i)(\omega ) = f^\sharp (x_i(\omega ), a_i^-(\omega )).\end{equation}
It is clear that the right hand side of that equation corresponds to the lower-left path in the commutative square; but it is not so clear that the left hand side of that equation corresponds to the upper-right path; we must show that
\begin{equation}(\Delta  x) (c^\sharp (i, \omega )) = P(x_{i+1}|\mathcal {F}_i)(\omega )\end{equation}
Expanding out the definition of \(c^\sharp \), we have
\begin{equation}(\Delta  x) (c^\sharp (i, \omega )) = (\Delta  x) (i+1, P(1_\Omega |\mathcal {F}_i)(\omega ))\end{equation}
Then using that \(x(i+1, \omega ) = x_{i+1}(\omega )\), and \cref{ocl-00CW}, the right hand side of this is equal to \(P(x_{i+1}| \mathcal{F}_i)(\omega)\), as required. We are done.
\end{proof}

\begin{remark}[{An alternative clock}]\label{rt-0013}
We have chosen the clock system \(\tilde {\Omega }\) so that we can make a precise connection to the classical theory in \cref{rt-000Q}. However, if we don't care so much about recovering precisely the classical definition, we could choose another clock system with state space \(\Omega ^\ast  = \sum _{n \in  \mathbb {N}} \Omega ^n\) for some fixed probability \(\Omega \).

The update for this clock system takes a list \((\omega _1,\ldots ,\omega _n)\) and appends a randomly sampled element of \(\Omega \) to form \((\omega _1,\ldots ,\omega _n,\omega _{n+1})\).

This is almost the same as the clock system with state space \(\sum _{n : \mathbb {N}} (\Omega ^\mathbb {N}, \mathcal {F}_n)\) where the \(\sigma \)-algebra \(\mathcal {F}_n\) is induced by the projection to the first \(n\) elements. However, technically speaking there is no isomorphism between \((\Omega ^\mathbb {N}, \mathcal {F}_n)\) and \(\Omega ^n\) in the category of measure spaces, even though they are isomorphic as locales.

The advantage of the clock system \(\Omega ^\ast \) is that it is somewhat easier to describe, and additionally it is closer to how a computer implementation would work (at each step, sample a new “seed”).

In \cref{rt-000X} we use a similar idea to this in order to get behaviors for nondeterministic systems.
\end{remark}
\section{Representable behaviors for discrete-time nondeterministic Moore machines}\label{rt-000V}
To our knowledge, there is not a well-established classical definition of behavior for discrete-time nondeterministic Moore machines.

However, a similar kind of clock can work in the nondeterministic setting as in the stochastic setting, and this provides some notion of behavior; we will leave it to the reader to ascertain if it is a useful notion.

\begin{definition}[{Nondeterministic linear clock system}]\label{rt-000X}
Let \(\Omega \) be a set, and define the following nondeterministic system on the interface \((\Omega ^\ast , 1)\). The state space is \(\Omega ^\ast \), and \(c \colon  T\Omega ^\ast  \mathrel {\mkern 3mu\vcenter {\hbox {$\shortmid $}}\mkern -10mu{\to }} (\Omega ^\ast , 1)\) is defined by
\begin{equation}c(w) = w\end{equation}\begin{equation}c^\sharp (w, \ast ) = \{[w\ldots , \omega ] \mid  \omega  \in  \Omega \}\end{equation}
Intuitively, the update function nondeterministically chooses a new \(\omega \) to add to the end of the list.
\end{definition}

As we do not know of a well-established classical definition of behavior, the equivalent of \cref{rt-000Q} would just be a definitional equality; we define a behavior to be a map out of a nondeterministic clock system.

In \cite{wang-2025-nondeterministic}, Wang considers not just linear nondeterministic behaviors, but also non-linear nondeterministic behaviors. Inspired by this we propose another kind of clock system parametrized by a graph.

\begin{definition}[{Nondeterministic nonlinear clock systems}]\label{rt-000Y}
Let \(G\) be a graph and \(\Omega \) a set. Then we may form a set \(\mathbf {Path}_\Omega (G)\) of \(\Omega \)-decorated paths in \(G\). Concretely, an element of \(\mathbf {Path}_\Omega (G)\) is a path in \(G\) with an element of \(\Omega \) for every edge in the path.

There is then a system \(c_{\Omega , G}\) on the interface \((\mathbf {Path}_\Omega (G), 1)\) with state space \(\mathbf {Path}_\Omega (G)\) defined by:
\begin{equation}c_{\Omega , G}(p) = p\end{equation}\begin{equation}c_{\Omega , G}^\sharp (p, \ast ) = \{[p\ldots ,(e,\omega )] \mid  \text {$e$ is an edge starting at the end of $p$, $\omega  \in  \Omega $}\}\end{equation}
Note that in the case that \(G\) is the graph with a single node and a single edge, we recover the clock system of \cref{rt-000X}.
\end{definition}
\section{Profiting from representable behaviors}\label{rt-0014}
The fact that behavior is representable is not just nice in the abstract, it means that several theorems about behavior become much easier to prove.
\subsection{Compositionality of behavior}\label{rt-0015}
In \cite{myers-2021-double} (and later in \cite{myers-2023-categorical}), Myers shows that given a clock system in any systems theory, the induced representable behavior functor satisfies a certain compositionality theorem. In this section, we simply explain what that means in the context of our stochastic clock system of \cref{rt-000P}.

We lift our behaviors from a functor \(\mathsf {Sys} \to  \mathsf {Set}\) to a morphism of systems theories from the systems theory of discrete-time stochastic Moore machines to the behavioral (or “Jan Willems-style”, see \cite{willems-2007-behaviour}) systems theory. We do not delve into the double category theory here, instead we simply explain how the behavior functor extends to the rest of the systems theory (interfaces and composition patterns) in concrete terms.

To recap, the behavioral systems theory is characterized by the following data.
\begin{itemize}\item{}An interface is a set \(A\).
\item{}An interface morphism is a function.
\item{}A composition pattern from \(A\) to \(B\) is a span:

  
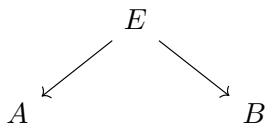
\begin{center}
\begin {tikzcd}
	& E \\
	A && B
	\arrow [from=1-2, to=2-1]
	\arrow [from=1-2, to=2-3]
\end {tikzcd}
\end{center}
\item{}A composition pattern morphism is a span morphism

  \begin{center}
\begin {tikzcd}
	{A_1} & {E_1} & {B_1} \\
	{A_2} & {E_2} & {B_2}
	\arrow [from=1-1, to=2-1]
	\arrow [from=1-2, to=1-1]
	\arrow [from=1-2, to=1-3]
	\arrow [from=1-2, to=2-2]
	\arrow [from=1-3, to=2-3]
	\arrow [from=2-2, to=2-1]
	\arrow [from=2-2, to=2-3]
\end {tikzcd}
\end{center}
\item{}A system on an interface \(A\) is a set \(X\) with a function \(X \to  A\), or in other words a span from \(1\) to \(A\).
\item{}A morphism of systems from \(X_1 \to  A_1\) to \(X_2 \to  A_2\) is a commutative square

  \begin{center}
\begin {tikzcd}
	{X_1} & {A_1} \\
	{X_2} & {A_2}
	\arrow [from=1-1, to=1-2]
	\arrow [from=1-1, to=2-1]
	\arrow [from=1-2, to=2-2]
	\arrow [from=2-1, to=2-2]
\end {tikzcd}
\end{center}\end{itemize}
Given a fixed probability space \((\Omega , \mathcal {F}_\Omega , P)\), morphisms out of the clock system \(T\tilde {\Omega } \to  (\tilde {\Omega }, 1)\) assemble into a systems theory map in the following way.
\begin{itemize}\item{}We send an interface \(A = (A^+, A^-)\) to the set of charts from \((\tilde {\Omega }, 1)\) to \(A\), \(\mathsf {Chart}((\tilde {\Omega }, 1), A)\).
\item{}An interface morphism \(a \colon  A_1 \to  A_2\) is sent to the post-composition function \(a \circ  - \colon  \mathsf {Chart}((\tilde {\Omega }, 1), A_1) \to  \mathsf {Chart}((\tilde {\Omega }, 1), A_2)\)
\item{}A composition pattern \(f \colon  A \mathrel {\mkern 3mu\vcenter {\hbox {$\shortmid $}}\mkern -10mu{\to }} B\) is sent to the relation (recall that relations are a subclass of spans) between \(\mathsf {Chart}((\tilde {\Omega }, 1), A)\) and \(\mathsf {Chart}((\tilde {\Omega }, 1), B)\) where \(a \colon  (\tilde {\Omega }, 1) \to  A\) and \(b \colon  (\tilde {\Omega }, 1) \to B\) are related if and only if the following commutes
  \begin{center}
\begin {tikzcd}
	{(\tilde {\Omega },1)} & {(\tilde {\Omega }, 1)} \\
	A & B
	\arrow [equals, from=1-1, to=1-2]
	\arrow ["a"', from=1-1, to=2-1]
	\arrow ["b", from=1-2, to=2-2]
	\arrow ["f"', "\shortmid "{marking}, from=2-1, to=2-2]
\end {tikzcd}
\end{center}This can be thought of as “the compatibility relation between behaviors on the interface.”
\item{}A morphism of composition patterns is sent to the span morphism given by postcomposition, just like a morphism of interfaces
\item{}A system \(f \colon  TX \to  A\) is sent to the set of system morphisms from \(c \colon  T\tilde {\Omega } \to  (\tilde {\Omega }, 1)\) to \(f\), with corresponding restriction map into \(\mathsf {Chart}((\tilde {\Omega }, 1), A)\). That is, a behavior of a system restricts to a behavior on the interface.
\item{}A system morphism is sent again to the action of postcomposition, just like a morphism of composition patterns\end{itemize}
The upshot of all of this is that as soon as we have defined behaviors representably, we automatically get a natural interaction between behaviors of systems and behaviors on their interfaces, and in addition a relationship (which we do not have space to discuss in depth here) between the action of composition patterns on systems and systems behaviors (such a relationship is often termed a “compositionality theorem for behaviors”). The reader is encouraged to refer to section 6 in \cite{myers-2021-double} for more details.
\subsection{Sheaves of behaviors}\label{rt-0016}
In \cref{rt-000A} we characterized behavior in terms of a fixed filtered probability space \((\Omega , \mathcal {F}, P)\). This is inelegant because there is of course not a canonical \(\Omega \); we would like to be able to vary \(\Omega \). In this section we show how to do this.

In \cite{simpson-2017-probability}, Simpson gives a notion of “probability sheaf”, which is a sheaf on the following site.

\begin{definition}[{Site of sample spaces}]\label{rt-0017}

Let \(\mathbb {P}\) be the category where the objects are pairs \((X, P)\) of a standard Borel space \(X\) equipped with a probability measure \(P\), and the morphisms are measurable maps that preserve the probability measure.

Then we can equip \(\mathbb {P}\) with the atomic topology, in which the covering sieves are simply the non-empty sieves; the resulting site is what we call the site of sample spaces.
\end{definition}

\begin{proposition}[{Random variables form a sheaf}]\label{rt-0018}
Given a Polish space \(A\), then the functor
\begin{equation}\underline {\mathrm {RV}}(A)(\Omega ) = \{x \colon  \Omega  \to  A \mid  \text {$x$ is Borel measurable}\}/=_{\text {a.e.}}\end{equation}
sending a sample space to the set of \(A\)-valued random variables quotiented by almost-everywhere equality is a sheaf on the site of sample spaces. This is proved in \cite{simpson-2017-probability}
\end{proposition}

	We will introduce a slight modification of this category, replacing probability spaces with filtered probability spaces.

\begin{definition}[{The category of filtrations}]\label{rt-1LAB}
  Let \(\mathbb {P}^\mathbb {N}_\mathrm {fil}\) denote the category where
  \begin{enumerate}\item{}Objects are standard Borel probability spaces \((\Omega ,\mathcal {F},P)\) equipped with a filtration \(\mathcal {F}_1 \subseteq  \mathcal {F}_2 \subseteq  \dots \) of \(\mathcal {F}\).
    \item{}Morphisms are measure-preserving, measurable functions \(f\) which preserve the filtration.\end{enumerate}
  Following \cite{simpson-2017-probability}, equip this with the coverage where every morphism is a singleton covering family.

\end{definition}

	We require that the ``full'' \(\sigma \)-algebra \(\mathcal {F}\) makes \((\Omega , \mathcal {F})\) a standard Borel space, but not that the filtration degrees \((\Omega , \mathcal {F}_i)\) are standard Borel (this is usually not possible). However, this hypothesis is enough to construct the conditional distributions we need, since we can always form the conditional \((\Omega , \mathcal {F}_i) \to  (\Omega , \mathcal {F})\) and then restrict down to the relevant filtration degree.

	Note that the assignment \((\Omega , \mathcal {F}, P) \mapsto  (\Omega ^\mathbb {N}, \mathcal {F}^\mathbb {N}, P^\mathbb {N}, \mathcal {F}^{\{0,\dots  i-1\}})\), which carries a probability space to its \(\mathbb {N}\)-fold product with itself (equipped with the natural filtration, where the \(i\)th filtration degree is the \(\sigma \)-algebra generated by the first \(i\) coordinates) is a functor \(\mathbb {P} \to  \mathbb {P}_\mathrm {fil}^\mathbb {N}\).

	We will need to restrict our attention to a certain subset of morphisms in \(\mathbb {P}_\mathrm {fil}^\mathbb {N}\), namely those which preserve the conditional expectations of stochastic processes, in the following sense:

\begin{definition}[{Immersion of filtrations}]\label{rt-MQP8}
  Let \(\phi : (\Omega , \mathcal {F}_i) \to  (\Gamma , \mathcal {F}'_i)\) be a measurable map between filtered probability spaces, compatible with the filtration. Then precomposition with \(\phi \) carries an adapted stochastic process on \(\Gamma \) to one on \(\Omega \). If this operation preserves martingales, we call \(\phi \) an \emph{immersion of filtrations} (or simply an immersion).

\end{definition}

\begin{proposition}[{Characterization of immersions}]\label{rt-41MK}
  Let \(\phi : (\Omega , \mathcal {F}, P \mathcal {F}_i) \to  (\Gamma , \mathcal {F}', P', \mathcal {F}'_i)\) be a map of filtered probability spaces. Suppose the conditionals distributions \((\Omega , \mathcal {F}_i) \to  (\Omega , \mathcal {F}_{i+1})\), \((\Gamma , \mathcal {F}_i) \to  (\Gamma , \mathcal {F}_{i+1})\) all exist (for example if \((\Omega , \mathcal {F}), (\Gamma , \mathcal {F}')\) are standard Borel). Then the following are equivalent:

\begin{enumerate}\item{}\(\phi \) is an immersion, that is for every adapted process \(X_i: \Gamma  \to  \mathbb {R}\), if \(X_i\) is a martingale, then so is \(X_i\phi  : \Omega  \to  \mathcal {F}_i\).
  
  \item{}
    The following diagram in \(\mathsf {Stoch}\) commutes for each \(i\), where the horizontal maps are the respective conditional distribution kernels.
  \end{enumerate}\begin{center}
  \begin {tikzcd}
	{(\Omega , \mathcal {F}_i)} & {(\Omega , \mathcal {F}_{i+1})} \\
	{(\Gamma , \mathcal {F}'_i)} & {(\Gamma ,\mathcal {F}'_{i+1})}
	\arrow ["c", from=1-1, to=1-2]
	\arrow ["\phi ", from=1-1, to=2-1]
	\arrow ["\phi ", from=1-2, to=2-2]
	\arrow ["c", from=2-1, to=2-2]
\end {tikzcd}
\end{center}
\end{proposition}
\begin{proof}
   First assume \(\phi \) preserves martingales. Let \(V \in  \mathcal {F}'_{i+1}\). Consider the stochastic process on \(\Gamma \) \(X_n := \mathbb {E}[1_V \mid  \mathcal {F}'_n] = P(V \mid  F'_n)\). By the tower property of conditional expectation, this is a martingale. Hence \(P(\phi (\omega ) \in  V \mid  \mathcal {F}_n)\) is a martingale on \(\Omega \). But after unwinding the definitions, this is exactly the claim that the above diagram commutes.

    Now assume this diagram commutes for each \(i\), and let \(X_n\) be a martingale on \(\Gamma \). By induction it suffices to prove that \(\mathbb {E}[X_{i+1}(\phi (\omega )) \mid  \mathcal {F}_i] = X_{i}(\phi (\omega ))\) (almost surely). Since \(X\) is a martingale, we can rewrite the right-hand side as \(\mathbb {E}[X_{i+1}(-) \mid  \mathcal {F}'_i] (\phi (\omega ))\). Now using the fact that conditional expectation is equal to the expectation of the conditional distribution, this is equal to the left-hand side using commutativity of the diagram.
\end{proof}

	The term \emph{immersion} is from the literature on stochastic processes (see e.g. \cite{grigorian-2023-enlargement}), where it is used to refer to the preceding condition in the case where \(\phi \) is the identity map between two filtrations of \(\sigma \)-algebras on the same set. Note that in the general case the conditional distributions involved in the second condition may not exist, which is why the first condition is used.

\begin{definition}[{The category of immersions}]
		Let \(\mathbb {P}^\mathbb {N}_\mathrm {im} \subseteq  \mathbb {P}^\mathbb {N}_\mathrm {fil}\) be the subcategory with the same objects, but containing only the immersions as morphisms. We equip this with (the restriction of) the same coverage as \(\mathbb {P}^\mathbb {N}_\mathrm {im}\)
\end{definition}

We can then use some abstract nonsense to make behaviors form a presheaf on \(\mathbb {P}^\mathbb {N}_\mathrm {im}\). Specifically, there is a functor \(c_{-}\) from \(\mathbb {P}^\mathbb {N}_\mathrm {im}\) to \(\mathsf {Sys}\) induced by the fact that, given a measure-preserving, filtration-preserving map of filtered probability spaces \(\Omega _1 \to  \Omega _2\), there is an induced map \(\tilde {\Omega _1} \to  \tilde {\Omega _2}\). If \(\phi \) is an immersion, then this diagram commutes:
\begin{center}
\begin {tikzcd}
	{T\Omega _1^\ast } & {(\Omega _1^\ast ,1)} \\
	{T\Omega _2^\ast } & {(\Omega _2^\ast ,1)}
	\arrow ["{c_{\Omega _1}}", "\shortmid "{marking}, from=1-1, to=1-2]
	\arrow [from=1-1, to=2-1]
	\arrow [from=1-2, to=2-2]
	\arrow ["{c_{\Omega _2}}"', "\shortmid "{marking}, from=2-1, to=2-2]
\end {tikzcd}
\end{center}
Thus, we can take the nerve of \(c_{-}\) to get a functor \(B \colon  \mathsf {Sys} \to \mathsf {Psh}(\mathbb {P}_\mathrm {im}^\mathbb {N})\), defined by
\begin{equation}B(f)(\Omega ) = \mathsf {Sys}(c_\Omega , f)\end{equation}
This is not quite a sheaf, since it does not respect almost sure equality. However, the sheafification of this presheaf is exactly given by quotienting by almost sure equality.

The following theorem relies on lemmas found in Appendix~\ref{supp-lemmas}.

\begin{theorem}[{Behaviors of stochastic system form a probability sheaf}]\label{rt-PRZ6}
The following hold.
\begin{enumerate}\item{}
    There is a functor \(c_{-}: \mathbb {P}^\mathbb {N}_\mathrm {im} \to  \mathsf {Sys}\) which carries a filtered probability space \((\Omega ,\mathcal {F},P)\) to the system \(T\tilde {\Omega } \xrightarrow {c_\Omega } (\tilde {\Omega }, 1)\) defined above.
  
  \item{}
    The nerve of \(c_{-}\) is a functor \(B: \mathsf {Sys} \to  \mathsf {Psh}(\mathbb {P}^\mathbb {N}_\mathrm {im})\), where \(B(s)(\Omega ,F,P)\) is the set of \(\Omega \)-parameterized stochastic processes which satisfy the law of \(s\).
  
  \item{}
        The sheafification of this presheaf is given by quotienting each set \(B(s)(\Omega ,F,P)\) by \(P\)-almost sure equivalence. Denoting this by \(\bar {B}(s)(\Omega ,F,P)\), \(\bar {B}\) forms a functor \(\mathsf {Sys} \to  \mathsf {Sh}(\mathbb {P}^\mathbb {N}_\mathrm {im})\)\end{enumerate}
\end{theorem}
\begin{proof} There is really nothing to prove in constructing the functor \(c_{-}\); the only nontrivial step is the claim that each \(\phi : \Omega  \to  \Gamma \) goes to a morphism of systems, but this is essentially the second equivalent condition in \cref{rt-41MK}. The second part is now just the previously-proven classification of behaviors by system morphisms.

  Note that for a presheaf \(F\) to be separated is precisely the condition that any sections \(x,y \in  F(\Omega )\) which agree on some subspace \(\Omega ' \hookrightarrow  \Omega \) of full measure are already equal. Hence clearly any map \(B(s) \to  F\) where \(F\) is a sheaf factors uniquely over \(\bar {B}(s)\). So it remains to prove that \(B(s)\) is a sheaf on \(\mathbb {P}^\mathbb {N}_\mathrm {im}\). It's clear that it's a presheaf, since pullback along a measure-preserving map preserves almost-sure equality. It is separated essentially by construction.

  To see that it satisfies gluing, let \(\phi : \Gamma  \to  \Omega \) be a measure-preserving map, and let \(x_i: \Gamma  \to  X\) be a stochastic process, satisfying the law of \(s\), which is \(\phi \)-invariant. Then form the independent pullback \(\Gamma  \otimes _\Omega  \Gamma \), noting that the two parallel maps \(\Gamma  \otimes _\Omega  \Gamma  \rightrightarrows  \Gamma \) are equalized by \(\phi \).

  Hence for that measure on the pullback, \(x_n(\gamma _1) = x_n(\gamma _2)\) almost surely. Hence for almost every \(\omega \), there exists some \(q_n(\omega ) \in  X\) so that when \(\gamma \) is drawn from the conditional distribution given \(\omega \), \(x_n(\gamma ) = q_n(\omega )\) with probability \(1\). To see \(q_n\) is measurable, note that for \(A \subset  X\) measurable, \(q_n(\omega ) \in  A\) if and only if \(\gamma  \in  x_n^{-1}(A)\) with probability 1, which is a measurable set since the conditional distribution is a measurable kernel. It is easy to see that \(q_n\) is the requisite factorization of \(x_n\).
\end{proof}

\section{Future work}\label{rt-000U}
We have several directions for future work, mainly aimed at developing clock systems for new systems theories.
\subsection{Representable behaviors for imprecise probability Moore machines}\label{rt-0010}
Imprecise probability is a combination of non-determinism and stochasticity. One approach to imprecise probability is via “infradistributions”, which are convex closed subsets of the set of distributions on a measurable space. However the “infradistribution monad” is not strong, which means that we cannot build a systems theory for it in the same way we do with the Giry or powerset monads. In \cite{liell-cock-2024-compositional}, Liell-Cock and Staton give an account of imprecise probability via a graded monad, which recovers some monoidal structure. In future work we would like to incorporate this graded monad into a systems theory and exhibit clock systems for that systems theory in the same vein as the clock systems we have presented here.
\subsection{Representable behaviors for continuous-time stochastic Moore machines}\label{rt-000Z}
We are attempting to build a systems theory for continuous-time stochastic Moore machines (which would essentially be “open stochastic differential equations”). The foundations for this are not yet in place, but we conjecture that the system that represents trajectories should look something like the following. Let \(\Omega \) be some space with a notion of continuity (exactly what notion is yet undetermined). Then consider
\begin{equation}\bar {\Omega } = \sum _{t \in  \mathbb {R}_{\geq  0}} C([0,t], \Omega )\end{equation}
where \(C([0,t], \Omega )\) is the space of continuous functions from \([0,t]\) to \(\Omega \). If we let \(\Omega  = \mathbb {R}\), then an appropriate update function would look something like extending a continuous function \(\gamma  \colon  [0,t] \to  \mathbb {R}\) to a continuous function \([0,t+s]\) by sampling a Brownian motion on \([t,t+s]\) that starts at \(\gamma (t)\). For \(\Omega  = \mathbb {R}^\mathbb {N}\), we could extend a continuous function \(\gamma  \colon  [0,t] \to  \mathbb {R}\) by sampling countably many Brownian motions.

A behavior of the continuous-time stochastic Moore machine should be given a map that takes a continuous function \(\gamma  \colon  [0,t] \to  \mathbb {R}\) (thought of as a path of a Brownian motion) and performs stochastic integration along it to get the state of the Moore machine at time \(t\).
\nocite{*}\bibliographystyle{plain}\bibliography{\jobname.bib}

\appendix

\section{Supplementary lemmas for probability sheaves}\label{supp-lemmas}

\begin{lemma}[{Commutation of Bayesian inverses}]\label{rt-2219}
  Consider a commutative diagram of probability spaces and deterministic morphisms

\begin{center}
  \begin {tikzcd}
	A & B \\
	C & D
	\arrow [from=1-1, to=1-2]
	\arrow [from=1-1, to=2-1]
	\arrow [from=1-2, to=2-2]
	\arrow [from=2-1, to=2-2]
\end {tikzcd}
\end{center}

  in any Markov category. If they exist, we may form the horizontal Bayesian inverses, giving a square

\begin{center}
  \begin {tikzcd}
	A & B \\
	C & D
	\arrow [from=1-1, to=2-1]
	\arrow [from=1-2, to=1-1]
	\arrow [from=1-2, to=2-2]
	\arrow [from=2-2, to=2-1]
\end {tikzcd}
\end{center}

  Assuming the conditionals exist, this square commutes almost certainly if and only if the original square displays the conditional independence \(B \perp _D C\). In particular, it commutes if and only if the analogous square involving instead the vertical Bayesian inverses commutes almost surely.
\end{lemma}

\begin{proof} Since \(D\) is a deterministic function of \(B\), this conditional independence is equivalent to the statement that the conditional distribution of \(C\) given \(B\) factors over \(B \to  D\). But then this factorization must be the conditional distribution of \(C\) given \(D\), which is the desired commutativity condition.
\end{proof}

	We will refer to a square of probability spaces which displays the conditional independence \(B \perp _D C\) as here as an \emph{independent square}. Note that these play an important role in the theory of \cite{simpson-2017-probability} as well.

\begin{lemma}[{Independent squares}]\label{rt-1LTT}
  In any Markov category \(\mathsf {C}\), for \(i=0,1\), let \(\Gamma _i \to  \Lambda _i \leftarrow  \Omega _i\) be a cospan of probability spaces,
  and let measure-preserving maps \(\Gamma _1 \to  \Gamma _0, \Omega _1 \to  \Omega _0, \Lambda _1 \to  \Lambda _0\) be given so that the two squares commute, and further so that these squares are both independent in the sense of \cref{rt-2219}.
  Suppose the independent pullbacks \(\Gamma _i \otimes _{\Lambda _i} \Omega _i\) both exist. Then the square

\begin{center}
\begin {tikzcd}
	{\Gamma _1 \otimes _{\Lambda _1} \Omega _1} & {\Gamma _0 \otimes _{\Lambda _0} \Omega _0} \\
	{\Gamma _1} & {\Gamma _0}
	\arrow [from=1-1, to=1-2]
	\arrow [from=1-1, to=2-1]
	\arrow [from=1-2, to=2-2]
	\arrow [from=2-1, to=2-2]
\end {tikzcd}
\end{center}

  is again independent.
\end{lemma}
\begin{proof}
  Consider this diagram:

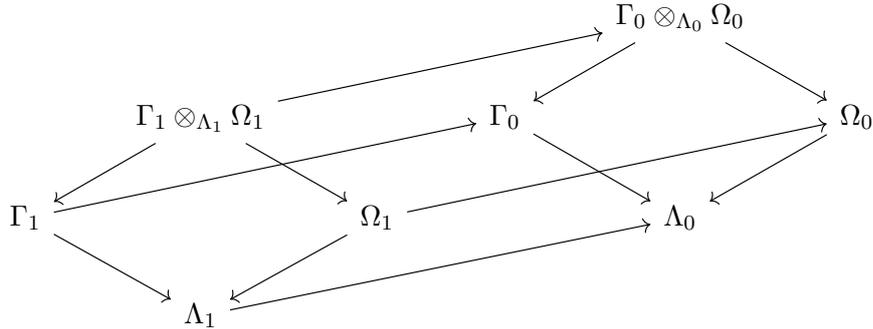
\begin{center}
  \begin {tikzcd}
	&&&& {\Gamma _0\otimes _{\Lambda _0}\Omega _0} \\
	& {\Gamma _1\otimes _{\Lambda _1}\Omega _1} && {\Gamma _0} && {\Omega _0} \\
	{\Gamma _1} && {\Omega _1} && {\Lambda _0} \\
	& {\Lambda _1}
	\arrow [from=1-5, to=2-4]
	\arrow [from=1-5, to=2-6]
	\arrow [from=2-2, to=1-5]
	\arrow [from=2-2, to=3-1]
	\arrow [from=2-2, to=3-3]
	\arrow [from=2-4, to=3-5]
	\arrow [from=2-6, to=3-5]
	\arrow [from=3-1, to=2-4]
	\arrow [from=3-1, to=4-2]
	\arrow [from=3-3, to=2-6]
	\arrow [from=3-3, to=4-2]
	\arrow [from=4-2, to=3-5]
\end {tikzcd}
\end{center}

  By \cref{rt-2219}, independent squares clearly compose, so that \(\Omega _0\) is independent of \(\Gamma _1\) given \(\Lambda _0\). But since \(\Gamma _0\) is a deterministic function of \(\Gamma _1\), this is equivalent to the independence expressed by the given square.
\end{proof}

\begin{lemma}[{Projections are immersions}]\label{rt-U1T3}
  Let \(\Omega _1 \to  \Omega  \leftarrow  \Omega _2\) be a cospan in \(\mathbb {P}^\mathbb {N}_\mathrm {im}\).
  If the pullback \(\Omega _1 \otimes _\Omega  \Omega _2\) is equipped with the natural filtration given in degree \(i\) by the intersection of \(\mathcal {F}_i^1 \otimes  \mathcal {F}_i^2\) and the diagonal, then the projection maps \(\Omega _1 \otimes _\Omega  \Omega _2 \to  \Omega _1, \Omega _2\) are immersions.
\end{lemma}
\begin{proof}
    It's clear that they are compatible with the filtration, so the only thing to verify is that they also commute with the conditional distributions. By \cref{rt-2219}, this is equivalent to checking that certain projection squares are independent, which follows from \cref{rt-1LTT}.
\end{proof}

\end{document}